\documentclass{amsart}
\usepackage[dvips]{graphicx}
\usepackage{amsmath,graphics,hyperref}
\usepackage{amsfonts,amssymb}
\usepackage[all]{xy}
\usepackage{comment}

\newtheorem{thm}{Theorem}

\newtheorem{prop}[thm]{Proposition}

\newtheorem{theorem}[thm]{Theorem}

\newtheorem{question}[thm]{Question}
\newtheorem{lemma}[thm]{Lemma}
\newtheorem{corollary}[thm]{Corollary}
\newtheorem{proposition}[thm]{Proposition}

\theoremstyle{definition}

\newtheorem{remark}[thm]{Remark}

\newtheorem*{claim1}{Claim 1}
\newtheorem*{claim2}{Claim 2}

\newcommand{\RP}{{\mathbb{RP}}{}^{3}}
\newcommand{\R}{\mathbb{R}}
\newcommand{\Z}{\mathbb{Z}}
\newcommand{\Q}{\mathbb{Q}}

\def \x {\times}

\def \sign{{\text{sign}}}

\def\G{\Gamma}

 \def\Q{\mathbb{Q}}  \def\Z{\mathbb{Z}} \def\R{\mathbb{R}} 
\def\N{\mathbb{N}}    
 \def\a{\alpha}   \def\bp{\begin{pmatrix}}
\def\sm{\setminus} \def\ep{\end{pmatrix}} \def\bn{\begin{enumerate}} 
   \def\en{\end{enumerate}}
\def\ba{\begin{array}} \def\ea{\end{array}} 
 \def\S{\Sigma}  \def\a{\alpha}  \def\ti{\tilde}
   \def\sign{\mbox{sign}}
  
\def\ker{\mbox{Ker}}\def\be{\begin{equation}} \def\ee{\end{equation}} 
   
 \def\hom{\mbox{Hom}}

\def\wti{\widetilde}

\def\co{\colon}
\def\wti{\widetilde}

\begin{document}

\title[Virtually symplectic fibered $4$-manifolds]
{Virtually symplectic fibered $4$-manifolds}

\author[R. \.{I}. Baykur]{R. \.{I}nan\c{c} Baykur}
\address{
Department of Mathematics and Statistics\\
 University of Massachusetts\\
  Amherst MA 01003-9305, USA}
 \email{baykur@math.umass.edu}

\author[S. Friedl]{Stefan Friedl}
\address{Fakult\"at f\"ur Mathematik\\ Universit\"at Regensburg\\93040 Regensburg\\   Germany}
\email{sfriedl@gmail.com}

\begin{abstract}
We mostly determine which closed smooth oriented $4$-manifolds fibering over lower dimensional manifolds are virtually symplectic, i.e. finitely covered by symplectic $4$-manifolds.
\end{abstract}

\maketitle

\setcounter{secnumdepth}{2}
\setcounter{section}{0}

\section*{Introduction}

Given a closed smooth oriented $4$-manifold $X$, one can ask the following:
\begin{enumerate}
\item [\textbf{Q1:}] Is $X$ symplectic, i.e.\,does it admit a closed non-degenerate $2$-form?
\item [\textbf{Q2:}] Is $X$ virtually symplectic, i.e.\,does it admit a finite cover that is symplectic?
\end{enumerate}
\smallskip
\noindent A well-known consequence of $X$ admitting a symplectic form is that it admits non-trivial Seiberg-Witten invariants by the work of Taubes \cite{Ta94,Ta95}, which in turn prevents $X$ or any finite cover from being  a connected sum of $4$-manifolds with $b^+>1$.

A finite cover of a symplectic $4$-manifold can always be equipped with a symplectic form, but the converse is false, i.e. not every $4$-manifold that is finitely covered by a symplectic $4$-manifold is necessarily symplectic itself.
For example, if $N$ is a non-fibered hyperbolic $3$-manifold, then by work of Agol \cite{Ag13} $N$ admits a finite cover $\ti{N}$ which is fibered, and it follows from \cite{Th76} that $S^1\times \ti{N}$ is symplectic whereas by \cite{FV11a} the manifold $S^1\times N$ is not symplectic. (Also see Proposition \ref{lem:abc} below.) Although these two questions are strongly interrelated, it is clear that one has more freedom when showing that $X$ is virtually symplectic and less when showing that it is actually symplectic.

In this paper we  will restrict ourselves to  $4$-manifolds which fiber over a smaller non-zero dimensional manifold.  The recent dramatic increase in our understanding of $3$-manifolds due to the work of Agol \cite{Ag08,Ag13}, Przytycki-Wise \cite{PW12} and Wise \cite{Wi12a,Wi12b} now makes it possible to settle \textbf{Q1} and \textbf{Q2}  for many such $4$-manifolds. For example \textbf{Q1} has been completely answered for $4$-manifolds which fiber over a $3$-manifold in \cite{FV11a,FV13} and for those which fiber over a surface in \cite{Th76,Geiges,Walczak}. To the best of our knowledge, the case of $4$-manifolds which fiber over $S^1$ has not been dealt with in any detail.
In this paper we will now mostly focus on \textbf{Q2}. Our goal is to determine when $X$ is virtually symplectic, in terms of the data provided by the fiber bundle structure, in particular by the topology of the base and the fiber.

We will say a $4$-manifold $X$ is \textit{fibered} if it admits a smooth bundle map $f\colon X\to Z$
where the fiber $Y$ and the base $Z$ are non-zero dimensional orientable smooth manifolds. We will often encode this information in the form $Y \hookrightarrow X \stackrel{f}{\rightarrow} Z$. Following the usual convention, we will call a $3$-manifold $Y$ \textit{fibered} if it fibers over\,$S^1$.

Our study naturally breaks down into three different cases corresponding to the possible dimensions of the base (and the complementary dimensions of the fiber), which we will take on in separate sections below. Our results can be reorganized and summarized in a pair of theorems as follows, which answer the existence question in the positive and in the negative, respectively:
\\

\noindent \textbf{Theorem A.} \emph{
Let $X$ be a  $4$-manifold and $A \hookrightarrow X \stackrel{f}{\rightarrow} B$ be a fiber bundle. $X$ is virtually symplectic if one of the following conditions is satisfied:
\begin{itemize}
\item[(1)] $B$ is an irreducible $3$-manifold which is not a graph manifold; \
\item[(2)] $A$ is a homologically essential surface in $X$, or $B$ is a $2$-torus;\
\item[(3)] $A$ is an irreducible $3$-manifold which has only hyperbolic pieces in its JSJ decomposition.
\end{itemize}
}
As a partial converse to Theorem~A we can now formulate the following theorem:

\noindent \textbf{Theorem B.}\emph{
Let $X$ be a  $4$-manifold and $A \hookrightarrow X \stackrel{f}{\rightarrow} B$ be a fiber bundle. $X$ is not virtually symplectic if:
\begin{itemize}
\item[(1)] $B$ is a not prime $3$-manifold or if $B$ is  a non-virtually fibered graph manifold;
\item[(2)] $A$ is homologically inessential and either $A$ or $B$ is not a torus;
\item[(3)] $A$ is a connected sum of non-spherical $3$-manifolds and the monodromy of $f$ preserves the separating $2$-sphere.
\end{itemize}
}
\noindent More generally, the statement of the theorem holds for any $X$ finitely covered by a fibered $4$-manifold as above.

\begin{remark} \noindent{\textbf{(1)}} We almost completely determine which $4$-manifolds fibering over a $3$-manifold are virtually symplectic. Our results make essential use of the recent work of Agol, Przytycki and Wise \cite{Ag08,Ag13, Wi12a,Wi12b, PW12}, and rely on \cite{FV13,FV12}. The only case we can not address completely is when the base manifold is a (virtually) fibered graph manifold: see the example and the discussion in Section \ref{section:example13}. \\
\noindent \textbf{(2)} The case of surface bundles over surfaces is an almost immediate consequence of the work by Thurston \cite{Th76}, Geiges \cite{Geiges} and Walczak \cite{Walczak}. In particular we will see that a $4$-manifold that fibers over a surface is virtually symplectic if and only if it is symplectic: see Theorem \ref{theorem22}. \\
\noindent \textbf{(3)} In \cite[Section 13.6]{Hillman}, Questions \textbf{Q1} and \textbf{Q2} for fibered $4$-manifolds are discussed by Hillman in terms of the \textit{$4$-manifold geometries}.
\end{remark}

We see that when $X$ admits a fiber bundle structure such that either the base or the fiber is an irreducible $3$-manifold $Y$, then $X$ is virtually symplectic, unless $Y$ has at least one Seifert fibered piece in its JSJ decomposition. Moreover, when $X$ fibers over a surface, it fails to be symplectic only if it is an $S^1$-bundle over an $S^1$-bundle over a surface, i.e. a trivial Seifert fibered $3$-manifold. Hence a heuristic conclusion we can make is that a fibered $4$-manifold $X$ fails to be virtually symplectic only if the fiber bundle structure ``contains a Seifert fibered $3$-manifold in an essential way'' as above.

\enlargethispage{1cm}
\vspace{0.1in}
\noindent \textit{Conventions.}  All manifolds in this article are assumed to be smooth, closed, connected and oriented, unless we say explicitly otherwise. Similarly, all covering and bundle maps are assumed to be smooth.
A graph manifold is an irreducible $3$-manifold such that all JSJ components are Seifert fibered spaces,
in particular a Seifert fibered space is a graph manifold.

\section{Fibering over a $3$-manifold}

In this section we will prove the following theorem which is precisely part (1) of Theorem~A and Theorem~B.

\begin{theorem}\label{theorem13}
Let $X$ be a  $4$-manifold which is a fiber bundle over a $3$-manifold $N$.  Then the following hold:
\bn
\item[(A)] $X$ is virtually symplectic if $N$ is prime and if $N$ is not a  graph manifold.
\item[(B)] $X$ is not virtually symplectic if $N$ is not virtually fibered, e.g. if $N$ is not prime.
\en
\end{theorem}

Theorem \ref{theorem13} settles the question which $S^1$-bundles over a $3$-manifold are virtually symplectic
except for the case that the base manifold is a virtually fibered graph manifold. In Section \ref{section:example13} we will see that the latter case is in fact more delicate and does not allow a simple statement. In particular we will see that there exists a fiber bundle $X$ over a virtually fibered graph manifold such that $X$ is not virtually symplectic.

Before moving on to the proof of Theorem~\ref{theorem13}, let us note the following corollary:

\noindent \textbf{Corollary.}\emph{
Let $X$ be a  $4$-manifold which is a fiber bundle over a $3$-manifold $B$. If $B$ is irreducible and not a graph manifold, then $X$ admits a finite cover $\ti{X}$ which is symplectic with $b_2^+(\ti{X})>1$.
}

Recall that this implies in particular that $\ti{X}$, and hence $X$, is not the connected sum of $4$-manifolds with $b^+>1$.

\begin{proof}
Let $B$ be a $3$-manifold which is irreducible and not a graph manifold
and let $X$ be a  $4$-manifold which is a fiber bundle over $B$.
We have to show that  $X$ admits a finite cover $\ti{X}$ which is symplectic with $b^+(\ti{X})>1$.

By Theorem \ref{theorem13}~(A) there exists a finite cover $X'$ of $X$ which is symplectic. Note that $X'$ is again an $S^1$-bundle over a $3$-manifold $B'$, where $B'$ is a finite cover of $B$. In particular $B'$ is a $3$-manifold which is irreducible and not a graph manifold. On the other hand, by Theorem \ref{thm:apw2}~(1) of Agol-Przytycki-Wise quoted below, the manifold $B'$ admits a finite cover $\ti{B}$ with $b_1(\ti{B})\geq 3$. We denote by $\ti{X}$ the cover of $X'$ corresponding to the cover $\ti{B}\to B'$.

The signature of any $S^1$-bundle over a $3$-manifold is zero, in particular \linebreak $\sign(\ti{X})=0$. It furthermore follows from the Gysin sequence that $b_2(\ti{X})=2b_1(\ti{B})$ (if the Euler class $e$ of the $S^1$-bundle $\ti{X}\to \ti{B}$ is torsion) or $b_2(\ti{X})=2b_1(\ti{B})-2$ (if  $e$  is non-torsion). In either case we conclude that
\[ b^+(\ti{X})=\frac{1}{2}b_2(\ti{X})\geq b_1(\ti{B})-1\geq 2.\]
Finally note that $\ti{X}$ inherits a symplectic structure from $X'$.
\end{proof}

\subsection{Symplectic $4$-manifolds and fibered $3$-manifolds}

Let $N$ be a 3-manifold and let $\phi\in H^1(N;\Z)=\hom(\pi_1(N),\Z)$. We say $\phi$ is a \emph{fibered class} if there exists a surface bundle map $p\colon N\to S^1$  such that $\phi=p_*\colon \pi_1(N)\to \Z=\pi_1(S^1)$.
We furthermore say that a rational class $\phi \in H^1(N;\Q)$ is fibered, if a non-zero integer multiple is fibered.
Note  that a $3$-manifold is fibered if and only if  it admits a fibered class.

We can now formulate the following theorem which is a key ingredient to the proof of Theorem \ref{theorem13}.

\begin{theorem}\label{thm:symp}
Let $N$ be a 3-manifold and let $p\colon X\to N$ be an $S^1$-bundle over $N$. We denote by $p_*\colon H^2(X;\Q)\to H^1(N;\Q)$ the map that is given by integration along the fiber. Then  $X$ is symplectic if and only if the image of $p_*$ contains a fibered class.
\end{theorem}

\begin{remark} \noindent \textbf{(1)} The `if' direction of the theorem was proved in \cite{Bou88,FGM91,FV12,Th76} and the `only if' direction was proved in  \cite{FV13}, extending earlier work in \cite{McCarthy,CM00,Et01,Bow09,FV11a,FV11b}. This theorem was recently extended to $4$-manifolds with a fixed point free $S^1$-action by Bowden \cite{Bow12}. \\
\noindent \textbf{(2)} Any finite cover of $X$ is again an $S^1$-bundle over a finite cover of $B$. Theorem \ref{theorem13}~(B) is therefore an immediate consequence of Theorem \ref{thm:symp}.
\end{remark}

Note that if  $X$ is of the form $S^1\times N$, then the map  $p_*\colon H^2(X;\Q)\to H^1(N;\Q)$ is an epimorphism.
We thus obtain the following special case of Theorem \ref{thm:symp}, which we record here for later reference:

\begin{corollary}\label{cor:sympproduct}
Let $N$ be a 3-manifold. Then  $X=S^1\times N$ is symplectic if and only if $N$ is fibered.
\end{corollary}

We can now show that the questions \textbf{Q1} and \textbf{Q2} in the introduction are indeed different questions.

\begin{proposition}\label{lem:abc} There exist fibered $4$-manifolds which are not symplectic but which are virtually symplectic.
\end{proposition}

\begin{proof} Let $N$ be a $3$-manifold which is not fibered but which is virtually fibered. (In Theorem \ref{thm:apw2} we will see that $N$ could in fact be any non-fibered hyperbolic $3$-manifold, but more `hands on' examples can also be given by graph manifolds.) It now  follows from Corollary \ref{cor:sympproduct} that  $X=S^1\times N$ is not symplectic but that $X$ is virtually symplectic.
\end{proof}

\subsection{The virtual fibering theorem} \label{section:agol}

Let $N$ be a $3$-manifold. In the following we  say that a class $\phi \in H^1(N;\Q)$ is \emph{quasi--fibered}  if  any neighborhood of $\phi$ in $H^1(N;\Q)$ contains a fibered class. Note that in particular fibered classes are quasi--fibered. We furthermore say that $\phi$ is \emph{virtually quasi--fibered} if $N$ admits a finite cover $p\colon N'\to N$ such that $p^*(\phi)\in H^1(N';\Q)$ is quasi--fibered.

We can now formulate the following theorem which is the second key ingredient to the proof of Theorem \ref{theorem13}.

\begin{theorem}\label{thm:apw2}\textbf{\emph{(Agol--Przytycki-Wise)}}
Let $N$ be an irreducible 3-manifold which is not a  graph manifold. Then the following hold:
\bn
\item Given any $k\in \N$ there exists a finite cover $\ti{N}$ of $N$ with $b_1(\ti{N})>k$.
\item Any  class in $H^1(N;\Q)$ is virtually quasi-fibered.
\en
\end{theorem}

\begin{proof}
Let $N$ be an irreducible 3-manifold which is not a  graph manifold.  If $N$ is hyperbolic then it follows from work of  Agol \cite{Ag13}, building on earlier work of Wise \cite{Wi09,Wi12a,Wi12b} that $\pi_1(N)$ is `virtually special' which implies that `$\pi_1(N)$ is virtually RFRS', see also \cite{AFW12} for detailed references.
(The precise meaning of `virtually RFRS' is of no concern to us.) The same conclusion in the non-hyperbolic case has been obtained  by Przytycki and Wise \cite{PW12}.

The fact that $\pi_1(N)$ is virtually RFRS  and not virtually solvable readily implies that given any $k\in \N$ there exists a finite cover $\ti{N}$ of $N$ with $b_1(\ti{N})>k$. We refer to \cite{Ag08} or \cite{AFW12} for details.
(Note that if $N$ is not hyperbolic, then conclusion (1) also follows from  earlier work of Luecke and Kojima \cite{Ko87,Lu88}.)

Furthermore Agol \cite[Theorem~5.1]{Ag08} (see also \cite[Theorem~5.1]{FK12}) showed that if $N$ is an irreducible 3-manifold such that $\pi_1(N)$ is virtually RFRS, then given any class in $H^1(N;\Q)$ there exists a finite cover $p\colon \ti{N}\to N$ such that $p^*\phi$ sits in the closure of a fibered face of the Thurston norm ball \cite{Th86}. This implies in particular that $\phi$ is virtually quasi-fibered.
\end{proof}

\begin{remark}
Let $N$ be an irreducible 3-manifold which is not a  graph manifold. The combination of statements (1) and (2) of the theorem implies  in particular that $N$ is virtually fibered.
\end{remark}

Let $N$ be an irreducible 3-manifold which is not a  graph manifold. It follows immediately from Corollary \ref{cor:sympproduct} and Theorem \ref{thm:apw2} that $X=S^1\times N$ is virtually symplectic, thus proving Theorem \ref{theorem13} (A) in the product case.

The case of a non-trivial $S^1$-bundle over $N$ is more delicate though. In Section \ref{section:thurstonnorm} we will first recall the main properties of the Thurston norm before we can give a proof of the general case of Theorem \ref{theorem13} (A) in Section \ref{section:theorem13}.

\subsection{The Thurston norm and fibered classes}\label{section:thurstonnorm}

Let $N$ be a 3-manifold and let $\phi\in H^1(N;\Z)$. It is well--known that  any class in $H^1(N;\Z)$ is dual to a properly embedded surface. Now recall that the  Thurston norm of $\phi$ is defined as
\[ x_N(\phi):= \min\{\chi_-(\S)\, |\, \S \subset N\mbox{ properly embedded and dual to }\phi\}.\]
Here, given  a surface $\S$ with connected components $\S_1\cup\dots \cup \S_k$
we define  its complexity by $\chi_-(\S):=\sum_{i=1}^k \max\{-\chi(\S_i),0\}$.
Thurston \cite{Th86} showed that  $x_N$ is a seminorm on $H^1(N;\Z)$ which thus can be extended to a seminorm on $H^1(N;\Q)$ which we also denote by $x_N$. Thurston furthermore proved that  the Thurston norm ball
\[ B(N):=\{ \phi \in H^1(N;\Q)\, | \, x_N(\phi)\leq 1 \}\]
 is a (possibly non--compact) finite convex polytope.

Moreover, Thurston \cite{Th86} showed that there exist open top--dimensional faces $F_1,\dots,F_r$ of $B(N)$
such that the set of fibered classes in $H^1(N;\Q)$ equals the union of the open cones on $F_1,\dots,F_r$.
The faces $F_1,\dots,F_r$ are referred to as the \emph{fibered faces} of $B(N)$.

Now let $p\colon \ti{N}\to N$ be a finite cover of degree $k$ and let $\phi\in H^1(N;\Q)$. Then
\be \label{equ:tnfinitecover} x_{\ti{N}}(p^*\phi)=k\cdot x_N(\phi),\ee
furthermore $ \phi$ is fibered if and only if $p^*\phi$ is fibered. We refer to  \cite[Corollary~6.18]{Ga83} for a proof of (\ref{equ:tnfinitecover}), while the statement regarding finite covers of fibered classes can be proved using Stallings' theorem \cite{St62}.

\subsection{The proof of Theorem \ref{theorem13} (A)}\label{section:theorem13}

We will now provide the proof of  Theorem \ref{theorem13}~(A). Let $N$ be an irreducible 3-manifold which is not a graph manifold and let $X$ be any $S^1$-bundle over $N$. We have to show that $X$ is virtually symplectic.

We denote by $e$ the Euler class of the $S^1$-bundle $X\to N$. If $e=0$, then $X=S^1\times N$ and, as we pointed out above, the theorem follows immediately from Corollary \ref{cor:sympproduct} and Theorem \ref{thm:apw2}.
If $e$  is torsion, then it is well-known that there exists a finite cover of $N$ such that the pull-back $S^1$-bundle has trivial Euler class. We refer to \cite[Proposition~3]{Bow09} and \cite[Proof~of~Theorem~2.2]{FV11c} for details. It thus follows again from Theorems \ref{thm:symp} and \ref{thm:apw2} that $X$ is virtually symplectic.

We now suppose that $e$ is non-torsion. We denote by $p_*\colon H^2(X;\Q)\to H^1(N;\Q)$ the map that is given by integration along the fiber. The Gysin sequence then says that the following is an exact sequence:
\be \label{equ:gysin} H^2(X;\Q)\xrightarrow{p_*} H^1(N;\Q)\xrightarrow{\cup e} H^3(N;\Q).\ee
Since $e$ is non-torsion it follows that $p_*(H^2(X;\Q))\subset H^1(N;\Q)$ is a codimension one subspace.

We now note that it follows from the work of Agol \cite{Ag13} that $N$ admits a finite cover such that the Thurston norm ball of the finite cover has at least three different  top dimensional faces. We refer to \cite[Proposition~8.12]{AFW12} for details. Since we only care about virtual properties we can now without loss of generality assume that  the Thurston norm ball of $N$ has at least three different top dimensional faces.

Since the Thurston norm ball of $N$ has at least three different top dimensional faces and since $p_*(H^2(X;\Q))$ is a codimension one subspace of $H^1(N;\Q)$ it now follows that $p_*(H^2(X;\Q))$ intersects a top-dimensional face $F$ of the Thurston norm ball of $N$. Note that the intersection is necessarily a codimension one intersection.
We pick a class $\phi$ in the intersection.

It now follows from Theorem \ref{thm:apw2}  that  there exists an $n$-fold regular cover $g\colon \ti{N}\to N$, such that $\ti{\phi}:=g^*\phi\in H^1(\ti{N};\Q)$ is quasi--fibered.

We denote by $f\colon \ti{X}\to X$ the induced cover. We thus get the following commutative diagram
\[ \xymatrix{ \ti{X}\ar[d]^f\ar[r]^{\ti{p}}& \ti{N}\ar[d]^g \\ X\ar[r]^p&N.}\]

We now consider the map \[ \frac{1}{n}g^*\colon  H^1(N;\Q)\to H^1(\ti{N};\Q).\]
It is well-known that this map is injective, and by (\ref{equ:tnfinitecover})  this map
is furthermore an isometry of vector spaces with a seminorm. In particular $\ti{F}:=\frac{1}{n}g^*(F)$ is a face of the Thurston norm ball of $\ti{N}$.

Since $\ti{\phi}:=g^*\phi\in H^1(\ti{N};\Q)$ is quasi--fibered it follows that $\ti{F}$ sits on the boundary of
a fibered face $\ti{G}$ of the Thurston norm ball of $\ti{N}$. Recall that the intersection of   $p_*(H^2(X;\Q))$ and $F$ is of codimension one in $F$. It follows from the above discussion that  the intersection of   $p_*(H^2(\ti{X};\Q))$ and $\ti{F}$ is also of codimension one in $\ti{F}$. Since $p_*(H^2(\ti{X};\Q))$ is of codimension one in $H^2(\ti{N};\Q)$ and since $\ti{G}$ is a top-dimensional face it now follows that the intersection of $p_*(H^2(\ti{X};\Q))$ with the \emph{interior} of  $\ti{G}$ is also non-trivial.

We thus showed that $p_*(H^2(\ti{X};\Q))$  contains a fibered class, which implies by Theorem \ref{thm:symp}
that $\ti{X}$ is symplectic. This concludes the proof of Theorem \ref{theorem13} (A).

\subsection{The case when the base manifold is $S^1\times \S$} \label{section:example13}

A key ingredient in the proof of Theorem \ref{theorem13} was that for the given $N$ any class in $H^1(N;\Z)$ is virtually quasi-fibered. This information on its own is not enough though to guarantee that any $S^1$-bundle over $N$ is virtually symplectic. Indeed, consider $N=S^1\times \S$ where $\S$ is a  surface. We denote by $c=[S^1\times *]\in H_1(S^1\times \S;\Z)$ the class represented by the $S^1$-factor. It is well-known that a class $\phi \in H^1(S^1\times \S;\Q)$  is fibered if and only if $\phi(c)\ne 0$. In particular any class in  $H^1(S^1\times \S;\Q)$ is quasi-fibered. Given $e\in H^2(S^1\times \S;\Z)$ we now denote by $p\colon X_e\to S^1\times \S$ the $S^1$-bundle which corresponds to the Euler class $e$.

We can now formulate the following lemma.

\begin{lemma}\label{lem:s1surface}
Let $N=S^1\times \S$ where $\S$ is a  surface and let $e\in H^2(S^1\times \S;\Z)$.
Then the following are equivalent:
\bn
\item $X_e$ is symplectic,
\item $X_e$ is virtually symplectic,
\item $e$ is not a non-zero multiple of $PD(c)\in H^2(N;\Z)$.
\en
\end{lemma}

\begin{proof}
Let  $e\in H^2(S^1\times \S;\Z)$. Note that the Gysin sequence implies that
\[ p_*(H^2(X_e;\Q))=\ker(\cup e:H^1(S^1\times \S;\Q)\to H^3(S^1\times \S;\Q)).\]
It follows easily that
\[ p_*(H^2(X_e;\Q)) \subset \{\phi\in H^1(S^1\times \S;\Q)\,|\, \phi(c)=0\}=\{\mbox{non-fibered $\phi$}\}\]
if and only if $e$ a non-zero multiple of $PD(c)$.
The equivalence of statements (1) and (3) is now an immediate consequence of Theorem \ref{thm:symp}.

It is obvious that (1) implies (2). It remains to show that (2) implies (3). Let $S^1\times F$ be a product manifold  where $F$ is a  surface and let $e\in H^2(S^1\times F;\Z)$. In the following we say that the $S^1$-bundle $X_e\to S^1\times F$ is \emph{good} if  $e$ is not a non-zero multiple of $PD(c)\in H^2(S^1\times F;\Z)$.

Let $\ti{X}$ be a finite cover of $X_e$. Note that $\ti{X}$ is naturally an $S^1$-bundle over a finite cover of $S^1\times \S$. It is  straightforward to see that the $S^1$-bundle $\ti{X}$ is good if and only if $X_e$ is good.

In particular, if $X_e$ admits a finite cover $\ti{X}$ which is  symplectic, then by the equivalence of (1) and (3) it follows that $\ti{X}$ is good, which by the above implies that $X_e$ is good.
\end{proof}

We conclude with the following question:

\begin{question}
Are there any other (virtually) fibered graph manifolds besides $S^1\times \S$ that admit $S^1$-bundles which are
not virtually symplectic?
\end{question}

\section{Fibering over a surface}

Let $F \hookrightarrow X \stackrel{f}{\rightarrow} B$ be a surface bundle over a surface with $F \cong \Sigma_g$ and $B \cong \Sigma_h$. It follows from a classical argument of Thurston's \cite{Th76} that if the homology class of the fiber $[F] \neq 0$ in $H_2(X ; \R)$, then $X$ can be equipped with a symplectic structure --- which moreover makes all fibers symplectic. The first Chern class of an almost complex structure associated to the fibration $f$ gives a class in  $H^2(X; \R)$ evaluating on $[F]$ as $\chi(F)=2-2g$, which implies that $X$ always admits a symplectic structure when $g \neq 1$.

When the fiber genus $g=1$, there are three cases to consider, depending on the base genus $h$ of the fibration:

(i) $h=0$: $X$ is either $S^2 \x T^2$, which can be equipped with the product symplectic form, or otherwise it is $S^1 \x S^3$ or $S^1 \x L(p,1)$; see \cite[Lemma~10]{BK}. The latter manifolds all have $b^+=0$, so do their finite covers, implying that none is virtually symplectic. \

(ii) $h=1$: This case is analyzed in detail by Geiges \cite{Geiges}, who showed that for all possible $T^2$-fibrations over $T^2$, the total space $X$ is symplectic, even if the fibration is not necessarily symplectic, i.e if the fiber is not homologically essential. \

(iii) $h \geq 2$: This last case was studied by Walczak in \cite{Walczak}, from which we can deduce that the following statements are equivalent:
\bn
\item  $X$ is not symplectic,
\item $[F] = 0$ in $H_2(X; \R)$,
\item $X$ is either an $S^1$-bundle over a non-trivial $S^1$-bundle over $B$, or $X$ is a non-trivial $S^1$-bundle over $S^1\times B$ such that $e(X)$ is a non-zero multiple of the Poincar\'e dual of $[S^1]\in H_1(S^1\times B;\Z)$.
\en

Here, the equivalence of (1) and (2) is stated in \cite[Theorem~4.9]{Walczak}.
We now show the  equivalence of (1) and (3).
First, if  the $T^2$-bundle over $B$ does not admit a free fiber preserving $S^1$ action,  then it follows from \cite[Proof~of~Theorem~4.9]{Walczak} that $X$ is symplectic, and it is clear that (3) does not hold. 
We can thus turn to the case that  $B$ admits a free fiber preserving $S^1$ action. Put differently,  $X$ is an $S^1$-bundle over a 3-manifold $N$, that in turn is an $S^1$-bundle over $B$.
If $N$ is a non-trivial $S^1$-bundle over $B$, then it follows from \cite[Theorem~3.1]{Walczak} that $X$ is not symplectic. So we are left with the case that $N=S^1\times B$. 
But this case was dealt with in Section~\ref{section:example13} above.

If $X$ is a manifold as described in (3), then any finite cover is also of the same type. This implies that if $X$ is not symplectic, then $X$ is also not virtually symplectic.

Combining all the results we have spelled out above, we get:

\begin{thm}\label{theorem22}
Let $F \to X \stackrel{f}{\rightarrow} B$ be a surface bundle over a surface with $F \cong \Sigma_g$ and $B \cong \Sigma_h$. Then
the following are equivalent:
\begin{enumerate}
\item $X$ is virtually symplectic,
\item $X$ is symplectic,
\item $g \neq 1$; or $g=1=h$, or $g=1\neq h$ and $[F] \neq 0 \in H_2(X;\R)$,
\item $[F]$ is a homologically essential surface in $X$, or $g=1$ and $h=1$.
\end{enumerate}
\end{thm}

\noindent We thus see that by  the work of Geiges and Walczak   we completely understand  which surface bundles over surfaces are virtually symplectic. In particular this theorem implies Theorem A (2) and Theorem B (2).

\section{Fibering over the circle}


Let $M \hookrightarrow X \stackrel{f}{\rightarrow} S^1$ be a fiber bundle where $M$ is a $3$-manifold. Equivalently, $X$ is the mapping torus
 of $M$ for some orientation preserving self-diffeomorphism $\phi$ of $M$. We will call the element in the mapping class group of $M$ represented by this self-diffeomorphism of $M$, the \emph{monodromy} of the fibration $f$.
It is well-known that the diffeomorphism type of $X$ only depends on the monodromy of the fiber bundle.

Our first observation is the following:

\begin{prop}\label{prop:finiteorder}
Let $M \hookrightarrow X \stackrel{f}{\rightarrow} S^1$ be a fiber bundle where $M$ is an irreducible  $3$-manifold. If the monodromy of $f$
has finite order in the mapping class group of $M$  and if $M$ is not a graph manifold, then $X$ is virtually symplectic.
\end{prop}

\begin{proof}
Let $\phi$ be an orientation preserving self-diffeomorphism of $M$ representing the monodromy of $f$.
By assumption, there is a finite iteration $\phi ^k = \text{id}_M$. Pulling back the bundle $f\colon  X \to S^1$ via the $k$-fold covering of $S^1$, we get another bundle $M \hookrightarrow \tilde{X} \stackrel{\tilde{f}}{\rightarrow} S^1$, where $\tilde{X}=M \times S^1$ is a finite cover of $X$. It then follows from Theorem A (1), applied to $A=S^1$ and $B=M$, that $\tilde{X}$ is virtually symplectic, so $X$ is virtually symplectic.
\end{proof}

From this proposition we obtain the following corollary which is precisely the statement of Theorem~A (3).

\begin{corollary}
Let $M$ be an irreducible $3$-manifold such that all JSJ pieces are hyperbolic, then any fiber bundle over $S^1$ with fiber $M$ is virtually symplectic.
\end{corollary}

\begin{proof} Let $M \hookrightarrow X \stackrel{f}{\rightarrow} S^1$ be a fiber bundle where $M$ is an irreducible $3$-manifold such that all JSJ pieces are hyperbolic. Let $\phi$ be an orientation preserving self-diffeomorphism of $M$ representing the monodromy of $f$. The JSJ tori of $M$ are unique up to isotopy, it follows that $\phi$ permutes the isotopy classes of the JSJ tori, in particular there exists a $k$ such that $\phi^k$ fixes the isotopy class of each JSJ torus. We now pick an orientation preserving self-diffeomorphism $\psi$ of $M$ which is isotopic to $\phi^k$
and which fixes each JSJ torus setwise. In particular $\psi$ restricts to a self-diffeomorphism of each JSJ component.

It is a consequence of Mostow rigidity that the mapping class group of a hyperbolic $3$-manifold is finite. It follows that there exists an $l$ such that $\psi^l$ is isotopic to the identity. We thus showed that the mapping class group of $M$ is finite.

It now follows from Proposition \ref{prop:finiteorder} that $X$ is virtually symplectic.
\end{proof}

We can now state and prove our final theorem, which should be regarded as an analogue of McCarthy's main theorem in \cite{McCarthy}. 

\begin{theorem} \label{mainthm4}
Let $M \hookrightarrow X \stackrel{f}{\rightarrow} S^1$ be a fiber bundle, where $M$ is a  $3$-manifold.  If $X$ is virtually symplectic, then $M$ does not admit a non-trivial decomposition $M=M_1 \# M_2$ preserved by the monodromy of $f$ where $M_1$ and $M_2$ are non-spherical $3$-manifolds.  
\end{theorem}


We will use the following lemma in the proof of our theorem, coming from Seiberg-Witten theory:

\begin{lemma} \label{SWargument}
Let $X$ be a $4$-manifold with $b^+>1$ splitting as $X = X_1 \cup X_2$, where $\partial X_1 = - \partial X_2 = \coprod M_j$ is a (disjoint) union of closed $3$-manifolds and $b^+(X_i)>0$. If each $M_j$ is a connected sum of spherical $3$-manifolds and $S^1 \x S^2$s, then $X$ cannot be symplectic. 
\end{lemma}

\begin{proof}
We can turn the splitting $3$-manifold into a connected one by tubing along embedded arcs in $X_2$ with end points on different $M_j$ components. It is easy to see that both new $4$-manifold pieces, obtained by \textit{adding $1$-handles} to $X_1$ and \textit{carving them} in $X_2$, still have $b^+>0$. One can see this for instance by noting that by general position any element in $H_2(X_2;\Z)$ can be represented by a surface disjoint from the chosen arcs. We therefore get a new splitting $X= \partial X'_1 \cup X'_2$ with $b^+(X'_i)>0$ such that
\[ \partial X'_1 = - \partial X'_2= S^3 / \, \Gamma_1 \# \cdots \# S^3 / \, \Gamma_k \# \, m \, S^1 \x S^2 \, , \] 
for some $k, m \geq 0$, with $\Gamma_i$, $i= 1, \ldots, k$, is a finite subgroup of $SO(4)$ acting on $S^3$ freely. Here we use the convention that $k=0$ is $\#_m S^1\times S^2$, and $k=m=0$ is $S^3$. Let us take this splitting instead, and label the components again as $X_1$ and $X_2$. 

The proof will now follow from a combination of fine results by Kronheimer-Mrowka and Taubes. In the case of $k=m=0$, the desired conclusion is a classical result: a connected sum of $X_i$ with $b^+(X_i)>0$ would have vanishing Seiberg-Witten invariants, which however is not possible for a symplectic $X$ with $b^+(X)>0$ by \cite{Ta94}. The case $k=0$ and $m=1$ is given in \cite[Proposition~1]{Bow09}, which we will generalize to any $k, m \geq 1$ (and any $\Gamma_i$) in what follows. 

As shown in \cite[Proposition~36.1.3]{KM}, for any connected $3$-manifold $Y$ admitting positive scalar curvature, the monopole Floer invariants $HM_{\bullet}(Y) = 0$. (Here $HM_{\bullet}$ is defined as the image of the $j$ map in the long exact sequence relating the three flavors of the monopole Floer groups, which is observed to be zero in Proposition~36.1.3.) Since all the summands of $Y$ listed above admit positive scalar curvature\footnote{Since a $3$-manifold with a $K(\pi,1)$ summand does not admit a positive scalar curvature \cite{GromovLawson2}, by prime decomposition and geometrization of $3$-manifolds, any closed $3$-manifold $Y$ admitting a positive scalar curvature would indeed be a connected sum of spherical $3$-manifolds and $S^1 \x S^2$s.}, so does $Y$ per connected sum being a codimension $\geq 3$ operation \cite{GromovLawson}. By \cite[Proposition 3.11.1(i)]{KM}, the vanishing of these Floer groups for the \textit{connected} $3$-manifold splitting $X$ into two pieces with $b^+>0$ as above implies that the Seiberg-Witten monopole invariant of $X$ is zero for any $Spin^c$ structure $\xi$. Here the monopole invariant of $\xi$ is defined as a sum of the Seiberg-Witten invariants of $Spin^c$ structures which differ from $\xi$ by torsion. 

Suppose that $X$ admits a symplectic structure $\omega$, then by \cite{Ta94}, for $\xi_0 = c_1(X, \omega)$, the canonical class, the Seiberg-Witten invariant evaluates as $SW(X, \xi_0) = \pm 1$. Furthermore, Taubes shows in \cite{Ta95} that there is no other $\xi \in Spin^c$ with $\xi - \xi_0$ torsion and $SW(X, \xi) \neq 0$. As observed in  \cite[Proposition~1]{Bow09}, this means that the vanishing of the monopole invariant above for $\xi_0$ would imply the vanishing of $SW(X, \xi_0)$, leading to a contradiction. This completes the proof. 
\end{proof}

\begin{proof}[Proof of Theorem~\ref{mainthm4}]
Let $M=M_1\# M_2$ be a $3$-manifold and let $\phi$  be an orientation-preserving diffeomorphism  which preserves the connected sum decomposition. This means that   there exist open balls $D_i\subset M_i$  and a separating $2$-sphere $S$ in $M$ so that
\[ M=(M_1 \setminus D_1) \cup_S (M_2 \setminus D_2),\]
and such that $\phi$  preserves the two parts $M_i\sm D_i$, $i=1,2$ and such that $\phi$ restricts to the identity on $S$.
For $i=1,2$ we henceforth write $M_i'=M_i\sm D_i$ and we denote the restriction of $\phi$ to $M_i'$ by $\phi_i$. Let $X$ be the mapping torus corresponding to $M$ and $\phi$.

We will first prove the following claim:

\begin{claim1}
There exists a  finite cover $q:N\to M$, an orientation-preserving diffeomorphism $\psi$ of $N$, and $n\in \N$
such that the following diagram commutes:
\be \label{equ:diagram} \xymatrix{ N\ar[d]_q\ar[r]^{\psi}& N\ar[d]^q \\ M\ar[r]^{\phi^n}& M}\ee
and such that the following hold:
\bn
\item $\psi$ restricts to the identity on $q^{-1}(S)$,
\item there exists a component $T$ of $q^{-1}(S)$ which separates $N$ into $N_1\# N_2$,
\item for $i=1,2$ the span of the components of $q^{-1}(S)\cap N_i$ has rank at least three in $H_2(N_i;\Z)$.
\en
\end{claim1}

Let $i\in \{1,2\}$. Since $\pi_1(M_i)$ is residually finite (see e.g. \cite{He87}) we can find an epimorphism $\a_i:\pi_1(M_i)\to G_i$,  onto a group of order at least ten. By a standard argument we can arrange that $\ker(\a_i)\subset \pi_1(M_i)$ is characteristic.
This implies that if $\wti{M}_i$ denotes the corresponding cover of $M_i$, then $\phi_i$ lifts to an
orientation-preserving diffeomorphism $\wti{\phi}_i$ of $\wti{M}_i':=\wti{M}_i\sm p_i^{-1}(D_i)$. We denote the degree of the covering map by $d_i$.
Note that $\wti{\phi}_i$ permutes the boundary components of $\wti{M}_i'$.
For an appropriately chosen $n_i\in \N$ we then see that $\wti{\phi}_i^{n_i}$ restricts to the identity on  each boundary component of $\wti{M}_i'$.
We now write $n=n_1n_2$ and  $\psi_i=\wti{\phi}_i^{n}$, $i=1,2$.

We then consider the disjoint union
\[ \underset{=:\wti{W}_1}{\underbrace{\bigcup_{i=1}^{2d_2+2} \wti{M}_1' }}\,\,\, \cup\,\,\,
\underset{=:\wti{W}_2}{\underbrace{\bigcup_{i=1}^{2d_1} \wti{M}_2' \,\cup\,\bigcup_{i=1}^{2d_1}M_2'}}.\]
Note that  $\wti{W}_1$  admits a $d_1(2d_2+2)$-covering $q_1:\wti{W}_1\to M_1'$ and that $\wti{W}_2$ admits a $d_1(2d_2+2)$-covering $q_2:\wti{W}_2\to M_2'$.
In particular both manifolds have precisely $d_1(2d_2+2)$ boundary components.

We now partition $\wti{W}_1$ into two parts $\wti{W}_{1}^a$ and $\wti{W}_{1}^b$ which each consist of $d_2+1$ components.
We then glue a copy of $\wti{M}_2'$ to $\wti{W}_1$ such that precisely one boundary component of $\wti{M}_2'$ gets glued to $\wti{W}_{1}^a$
and all other components get glued to $\wti{W}_{1}^b$. We refer to this one boundary component henceforth as $T$.
Finally we glue the remaining components of $\wti{W}_2$ either to $\wti{W}_{1}^a$ or to $\wti{W}_{1}^b$
and we denote the resulting closed manifold by $N$.

Note that $T$ is a separating sphere in $N$. Also note that we can arrange the above gluings such that $N$
is connected.
We now summarize the properties of $N$:
\bn
\item there exists a covering map $q:N\to M$ which restricts to $q_i:\wti{W}_i\to M_i$ for $i=1,2$,
\item there exists a component $T$ of $q^{-1}(S)$ separating $N$ into two components $N_1$ and $N_2$,
\item there exists a self-diffeomorphism $\psi$ of $N$ which restricts to $\psi_1,\psi_2$ and $\phi_2^n$ on the components of $\wti{W}_1$ and $\wti{W}_2$, this self-diffeomorphism turns (\ref{equ:diagram}) into a commutative diagram,
\item for $i=1,2$ the number of components of $q^{-1}(S)$ in $N_i$ is at least $d_1d_2$.
\en
 We are left with showing that for $i=1,2$, the components of $q^{-1}(S)$ which lie in $N_i$ span a subspace of $H_2(N_i;\Z)$ of rank at least three. Given $i=1,2$ we denote by $\G_i$ the graph which has one vertex for each component of $\wti{W}_1$
which lies in $N_i$ and one  vertex for each component of $\wti{W}_2$ which lies in $N_i$ and one edge for each component $q^{-1}(S)$ which lies in $N_i$. We view these vertices and edges as a graph  with the obvious  attaching map. Note that $\G_i$ is connected and that
\[ \ba{rcl} b_1(\G_i)=1-\chi(\G_i)&=&1-\# \mbox{vertices in $\G_i$}+\#\mbox{edges in $\G_i$}\\
&\geq &1-((2d_2+2)+4d_1)+d_1d_2\geq 3,\ea \]
where we now used that $d_1$ and $d_2$ are at least ten.
Note that we have a canonical projection map $g_i:N_i\to \G_i$. It is obvious that $p^* \colon H^1(\G_i;\Z)\to H^1(N_i;\Z)$ is injective and that its image is exactly the subspace of $H_2(N_i;\Z)$ which is the span of the Poincar\'e duals of the components of $p^{-1}(S)$. We thus showed that the components of $p^{-1}(S)$ span a subspace of $H_2(N_i;\Z)$ of rank at least three. This concludes the proof of the claim.

We now denote by $Y$ the mapping torus of $(N,\psi)$. Note that $Y$ is a finite cover of the mapping torus of $(M,\phi^n)$, which in turn is a finite cover of $X$. We now denote by $Y_i$, $i=1,2$ the mapping tori of $(N_i,\psi|_{N_i})$. Note that $Y_1$ and $Y_2$ are precisely the components of $Y$ cut along $S^1\times T$.

\begin{claim2}
For $i=1,2$ we have $b_2^+(Y_i)\geq 1$.
\end{claim2}

Let $i\in \{1,2\}$. We denote the restriction of $\psi$ to $N_i$ by $\psi_i$.
The Mayer-Vietoris sequence corresponding to the decomposition of $N_i$ arising from writing the base $S^1$ of the bundle as the union of two intervals gives us a long exact sequence
\[ \ldots \to 0 \to H_3(Y_i;\Z)\to H_2(N_i;\Z)\xrightarrow{\psi_i-\operatorname{id}}H_2(N_i;\Z)\to H_2(Y_i;\Z)\to \ldots \]
We denote by $V$ the subspace of $H_2(N_i;\Z)$ spanned by the components of $q^{-1}(S)$ lying in $N_i$.
Note that $V$ is invariant under $\psi_i$ and that it has rank at least three.
It follows from the long exact sequence that $b_3(Y_i)\geq 3$.

On the other hand the signature and Euler characteristic of any bundle over $S^1$ are zero. It follows now easily from $b_3(Y_i)\geq 3$ that $b_2^+(Y_i)\geq 2$. This completes the proof of our claim.

We now suppose that $X$ is symplectic. It follows immediately that the finite cover  $Y$ also admits a symplectic structure $\omega$. In the above discussion we showed that $Y$ has $b^+(Y)> 1$, and decomposes as $Y=Y_1 \cup Y_2$ with $\partial Y_1 = - \partial Y_2 = S^1 \x S^2$, for $b^+(Y_i) > 0$. This however is impossible by our Lemma~\ref{SWargument} above. 

Finally we have to consider the case that $X$ is only covered by a symplectic $4$-manifold $\tilde{X}$. For $Y \to X$ the covering we constructed above, let  $p\co \tilde{Y} \to Y$ be the pull-back of $\tilde{X}\to X$.
We write $\tilde{Y}_i=p^{-1}(Y_i)$, $i=1,2$. Note that $p^{-1}(S^1\times S^2)$ is a (possibly) disjoint union of copies of $S^1\times S^2$. We now pick an identification $\partial \tilde{Y}_i=\#_m S^1\times S^2$ for an appropriate $m$, which separates $\tilde{X}$ with $b^+(\tilde{X}) > 1$ as $\tilde{X} = \tilde{X}_1 \cup \tilde{X}_2$ and $b^+(\tilde{X}_i)$. Once again, this is impossible by Lemma~\ref{SWargument}. 
\end{proof}

\begin{remark}
In the above proof, our assumption on each component $M_i$ being non-spherical was only used to get an epimorphism from $\pi_1(M_i)$ onto a finite group of order at least $10$, for $i=1,2$. This is in fact satisfied by most spherical $3$-manifolds as well, but not all, and some condition on finite quotients is necessary: for example if $M_i= \RP$, then $X=S^1 \x ( \RP \# \RP)$ is finitely covered by $S^1\times S^1\times S^2$ which is symplectic.
\end{remark}

\enlargethispage{1cm}
\vspace{0.1in}
\noindent \textbf{Acknowledgements.} The first author was partially supported by the NSF grant DMS-0906912. We would like to thank Stefano Vidussi for an insightful suggestion we have employed in the proof of Theorem~\ref{mainthm4} and Jonathan Bloom for a reference. We also would like to thank Ian Hambleton and Mayer Landau for helpful feedback. Finally we also wish to express our gratitude to the referee for carefully reading an earlier version of this paper and for giving very helpful feedback.


\begin{thebibliography}{111}



\bibitem{Ag08}
I. Agol, {\em Criteria for virtual fibering}, Journal of Topology 1 (2008), 269--284.



\bibitem{Ag13}
I. Agol, {\em The virtual Haken conjecture}, with an appendix by I. Agol, D. Groves and J. Manning, 
	Documenta Math. 18 (2013) 1045--1087.

\bibitem{AFW12}
M. Aschenbrenner, S. Friedl and H. Wilton, {\em 3-manifold groups}, preprint (2012), arXiv:1205.0202.


\bibitem{BK} R. I. Baykur and S. Kamada, {\em Classification of broken Lefschetz fibrations with small fiber genera,} Journal of the Mathematical Society of Japan (to appear), arxiv.org/abs/1010.5814.




\bibitem{Bou88}
A. Bouyakoub, {\em  Sur les fibr\'es principaux de dimension $4$, en tores, munis de structures symplectiques
invariantes et leurs structures complexes}, C. R. Acad. Sci. Paris S\'er. I Math. 306 (1988),  no. 9, 417--420.

\bibitem{Bow09} J. Bowden, {\em The topology of symplectic circle bundles}, Trans. Amer. Math. Soc.  361,  no. 10 (2009), 5457--5468.

\bibitem{Bow12} J. Bowden, {\em Symplectic 4-manifolds with fixed point free circle actions}, preprint (2012),
arXiv:1206.0458. To appear in Proc. Amer. Math. Soc.

\bibitem{CM00}
W. Chen and R. Matveyev, {\em Symplectic Lefschetz fibrations on $S^1\times M^3$},  Geom. Topol. 4 (2000), 517--535.

\bibitem{Et01}
T. Etg\"u, {\em Lefschetz fibrations, complex structures and Seifert fibrations on $S^1 \times M^3$},
Algebraic \& Geometric Topology Volume 1 (2001), 469--489.

\bibitem{FGM91} M. Fern\'andez, A. Gray and J. W. Morgan, {\em Compact symplectic manifolds with free circle actions, and Massey products}, Michigan Math. J. 38, no. 2 (1991), 271--283.


\bibitem{FK12}
S. Friedl and T. Kitayama, {\em The virtual fibering theorem for $3$-manifolds}, Preprint (2012), 
to be published by l'Enseignement Math\'ematique.


\bibitem{FV11a} S. Friedl and  S. Vidussi, {\em Twisted Alexander polynomials detect fibered 3-manifolds}, Annals of Math. 173 (2011), 1587--1643.

\bibitem{FV11b} S. Friedl and  S. Vidussi,  {\em Twisted Alexander polynomials and fibered 3-manifolds},
Low-Dimensional and Symplectic Topology, Proc. Sympos. Pure Math. 82 (2011), 111--130.

\bibitem{FV11c} S. Friedl and  S. Vidussi, {\em 
Symplectic 4-manifolds with K=0 and the Lubotzky alternative},
Math. Res. Let. 18 (2011), 513--519.

\bibitem{FV12} S. Friedl and S. Vidussi, {\em Construction of symplectic structures on $4$--manifolds with a free circle action},  Proc.  Royal Soc. of Edinburgh  142 (2012), 359--370.


\bibitem{FV13} S. Friedl and  S. Vidussi, {\em A vanishing theorem for twisted Alexander polynomials with applications to symplectic 4-manifolds},
J. Eur. Math. Soc. 15 (2013), no. 6, 2127--2041. 

\bibitem{Ga83} D. Gabai, {\em Foliations and the topology of 3-manifolds}, J. Differential Geometry 18 (1983), 445--503.

\bibitem{Geiges} H. Geiges, {\em Symplectic structures on $T^2$-bundles over $T^2$,} Duke Math. J. 67 (1992), 539--555.


\bibitem{GromovLawson} M. Gromov and H.B. Lawson, {\em The classification of simply connected manifolds of positive scalar curvature,} Ann. of Math. (2) 111 (1980), no. 3, 423--434.

\bibitem{GromovLawson2} M. Gromov and H.B. Lawson, {\em Positive scalar curvature and the Dirac operator on complete  Riemannian manifolds,} Inst. Hautes Etudes Sci. Publ. Math. No. 58 (1983), 83--196.

\bibitem{He87}
J. Hempel, {\em Residual finiteness for $3$-manifolds}, Combinatorial group theory and topology (Alta, Utah, 1984), pp. 379--396, Ann. of Math. Stud., 111, Princeton Univ. Press, Princeton, NJ, 1987.

\bibitem{Hillman} J. Hillman, {\em Four-manifolds, geometries and knots,} Geometry \& Topology Monographs, 5. Geometry \& Topology Publications, Coventry, 2002. 

\bibitem{Ko87}
S. Kojima, {\em Finite covers of $3$-manifolds containing essential
surfaces of Euler characteristic $=0$}, Proc. Amer. Math. Soc. 101 (1987),
no. 4, 743--747.

\bibitem{KM} P. Kronheimer and T. Mrowka, {\em Monopoles and three-manifolds,} Cambridge Univ. Press, 2007. 

\bibitem{Lu88} J. Luecke, {\em Finite covers of $3$--manifolds containing essential tori}, Trans. Amer. Math. Soc. 310  (1988), 381--391.

\bibitem{McCarthy} J. D. McCarthy, {\em On the asphericity of a symplectic $M^3 \x S^1$,} Proc. Amer. Math. Soc. 129 (2001), no. 1, 257–-264.


\bibitem{PW12} P. Przytycki and D. Wise, {\em Mixed 3-manifolds are virtually special}, Preprint (2012).

\bibitem{St62}
J. Stallings, {\em On fibering certain 3-manifolds}, 1962 Topology of 3-manifolds and related
topics (Proc. The Univ. of Georgia Institute, 1961) pp. 95--100 Prentice-Hall, Englewood Cliffs,
N.J. (1962)

\bibitem{Ta94} C. H. Taubes, {\em The Seiberg-Witten invariants and symplectic forms},
Math. Res. Lett. 1 (1994), 809--822.

\bibitem{Ta95} C. H. Taubes, {\em More constraints on symplectic forms from Seiberg-Witten invariants},
Math. Res. Lett. 2 (1995), 9--13.

\bibitem{Th76} W. Thurston,  {\em Some simple examples of symplectic manifolds,} Proc. Amer. Math. Soc. 55 (1976), 467-–468.

\bibitem{Th86} W. P. Thurston, {\em A norm for the homology of 3-manifolds}, Mem. Amer. Math. Soc. 339 (1986), 99--130.

\bibitem{Walczak} R. Walczak, {\em Existence of symplectic structures on torus bundles over surfaces,} Ann. Global Anal. Geom., 28:3 (2005), 211-–-231.

\bibitem{Wi09}
D. Wise, {\em The structure of groups with a quasiconvex hierarchy}, Electronic Res. Ann. Math. Sci.  16 (2009), 44--55.

\bibitem{Wi12a}
D. Wise, {\em The structure of groups with a quasiconvex hierarchy}, 181 pages, Preprint  (2012)\\
downloaded on October 10, 2012 from the webpage for the NSF-CBMS conference.


\bibitem{Wi12b}
D. Wise, {\em From riches to RAAGs: $3$-manifolds, right--angled Artin groups, and cubical geometry}, CBMS Regional Conference Series in Mathematics, 2012.

\end{thebibliography}
\end{document}